\documentclass[11pt]{amsart}
\usepackage{amsmath}
\usepackage{amsfonts}
\usepackage{amssymb}
\usepackage{graphicx}

\newcommand\R{{\mathbb{R}}}

\theoremstyle{plain}
  \newtheorem{theorem}[subsection]{Theorem}
  
  \newtheorem{proposition}[subsection]{Proposition}
  \newtheorem{lemma}[subsection]{Lemma}

\newtheorem{mainthm}{Theorem}

\theoremstyle{remark}

\theoremstyle{definition}

\begin{document}
\title[]{On smoothness of extremizers of the Tomas-Stein inequality for $S^1$}
\author{Shuanglin Shao and Ming Wang}
\address{Department of Mathematics, KU, Lawrence, KS 66045}
\email{slshao@ku.edu, mwang@ku.edu}

\vspace{-1in}
\begin{abstract}
We prove that the extremizers to the Tomas-Stein inequality for the one dimension sphere are smooth. This is achieved by studying the associated generalized Euler-Lagrange equation. The goal is to establish real or complex analyticity of extremizers. 
\end{abstract}
\maketitle

\section{Introduction}
To understand the Fourier transform of functions on the Euclidean space, Stein \cite{Stein:1993} proposed the restriction problem. Let $d\ge 1$ be a fixed integer. Let $S$ be a smooth compact hypersurface with boundary in the space $\mathbb{R}^{d+1} =\mathbb{R}\times \mathbb{R}^d$ and $\sigma$ be the induced Lebesgue measure on $S$. Stein's restriction problem asks for which $1\le p,q\le \infty$ is the following estimate true,
$$ \|\widehat{f\sigma}\|_{L^q(\mathbb{R}\times \mathbb{R}^d)} \le C_{p,q,d,S} \|f\|_{L^p(S)},$$
for all test functions, where $\widehat{F}$ is the space time Fourier transform. It is not hard to see that $p,q$ satisfy the following necessary conditions
$$ q>\frac {2(d+1)}{d},\,\frac {d}{p'} \ge \frac {d+2}{q},$$
where $p'$ is the conjugate exponent of $p$. This problem is related to several outstanding conjectures in harmonic analysis such as the Bochner-Riesz conjecture and the Kakeya conjecture; for the references, see for instance \cite{Stein:1993, Bourgain:1991:Besicovitch-maximal, Bourgain-Guth:2011:bounds-oscillatory-integrals, Guth:2016:restriction-polynomial-partition, Guth:2016:restriction-polynomial-partition-II, Tao:2003:paraboloid-restri, Wolff:2001:restric-cone}. 

Let $S=S^d$, the unit sphere in $\mathbb{R} \times \mathbb{R}^d$, and $\sigma$ be the surface measure. The Tomas-Stein inequality for the sphere is
\begin{equation}\label{TS}
\|\widehat{f\sigma}\|_{L^{2+\frac 4d}(\R^{d+1})}\le \mathcal{R}_d \|f\|_{L^2(S^d)}
\end{equation}
where $d\ge 1$, $\mathcal{R}_d$ denotes the optimal constant
\begin{equation}\label{eq-1}
\mathcal{R}_d = \sup_{\substack{f\in L^2\\ f\neq 0 }}\frac {\|\widehat{f\sigma}\|_{L^{2+\frac 4d}} }{ \|f\|_{L^2(S^d)}}.
\end{equation}
The Tomas-Stein inequality \label{TS} belongs to the family of Fourier restriction inequalities. It can be regarded as an estimate of the Fourier transform of a measure supported on the sphere $S^d$. The non-endpoint estimate was established by Tomas while the endpoint was established by Stein by complex interpolation \cite{Stein:1993}. Its variants, the Strichartz inequalities, are useful in the partial differential equations, see for instance  \cite{Tao:2006-CBMS-book}. 

Recently the study of the extremal problem for the Fourier restriction inequality or the Strichartz inequalities has attracted a lot of attention. It includes the questions of proving existence of extremizers and establishing characterization of extremizers such as regularity or uniqueness for these inequalities. 

A variant of \eqref{TS} is the Strichartz inequality for the Schr\"odinger equation, 
\begin{equation}\label{eq-Strichartz}
\|e^{it\Delta} f \|_{L^{2+\frac 4d}_{t,x}(\mathbb{R}\times \mathbb{R}^d)} \le C \|f\|_{L^2(\mathbb{R}^d)}
\end{equation}
where $e^{it\Delta} f (x) = \frac {1}{(2\pi)^{d/2}}\int_{\mathbb{R}^d} e^{ix\cdot \xi +it |\xi|^2} \hat{f}(\xi) d\xi$ and $d\ge 1$. This can be viewed as an estimate of the Fourier transform of a measure supported on the paraboloid in $ \mathbb{R}\times \mathbb{R}^d$. For $d=1$ in \eqref{eq-Strichartz}, Kunze is the first to prove the existence of extremizers in \cite{Kunze:2003:maxi-strichartz-1d} by an elaborate concentration-compactness argument. Foschi \cite{Foschi:2007:maxi-strichartz-2d}, Hundertmark and Zharnitsky \cite{Hundertmark-Zharnitsky:2006:maximizers-Strichartz-low-dimensions} show that Gaussian functions are the only extremizers for \eqref{eq-Strichartz} when $d=1,2$. Bennett, Bez, Carbery and Hundertmark \cite{Bennett-Bez-Carbery-Hundertmark:2008:heat-flow-of-strichartz-norm} give a different proof of this fact by using the heat-flow method. In \cite{Carneiro-2009-sharp-strichartz}, Carneiro establishes some sharp Strichartz inequalities for the Schr\"odinger equation. When $d\ge 3$, we \cite{Shao:2009:maximizers-schrodinger} have proved the existence of extremizers by using the profile decompositions developed in \cite{Begout-Vargas:2007:profile-schrod-higher-d}. For the wave equation, Bulut \cite{Bulut:2009:maximizer-wave} has proved the existence of extremizers by using the profile decompositions in the spirit of \cite{Bahouri-Gerard:1999:profile-wave}.

An extremizer $f$ to the Tomas-Stein inequality \eqref{TS} is a nonzero function $f\in L^2$ such that 
$ \|\widehat{f\sigma}\|_{L^{2+\frac 4d}(\R^{d+1})}= \mathcal{R}_d \|f\|_{L^2(S^d)} .$ In this note, we specify the dimension $d=1$ and write $\mathcal{R} = \mathcal{R}_1$. In \cite{Shao:2015:extremals-1d-sphere}, we have proved there exists an extremizer when $d=1$. Here we establish the smoothness property of extremizers. The work  \cite{Shao:2015:extremals-1d-sphere} and this note follow roughly similar lines as in \cite{Christ-Shao:extremal-for-sphere-restriction-I-existence, Christ-Shao:extremal-for-sphere-restriction-II-characterizations}. In the previous work \cite{Christ-Shao:extremal-for-sphere-restriction-I-existence, Christ-Shao:extremal-for-sphere-restriction-II-characterizations},  Christ and the author prove the existence of extremizers and established some characterization for the Tomas-Stein inequality \eqref{TS} when $d=2$. In this case, Foschi \cite{Foschi:2013:global-extremisers-2d} settles down the problem by proving that constants are the only extremizers up to the complex modulation. In \cite{Carneiro-Foschi-Silva-Thiele:2015:sharp-trilinear-inequality-S1}, for \eqref{TS} when $d=1$, Carneiro, Foschi, Silva and Thiele recently prove a conditional result that constants are the only extremizers up to the complex modulation. This relies on the earlier work of Silva and Thiele \cite{Silva-Thiele:2015:integrals-6-Bessel-functions} about the inequality of a 6-fold product of Bessel functions and the study of a functional equation of Cauchy-Pexider type on the sphere in Charalambides \cite{Charalambides:2012:Cauchy-Pexider}. Frank, Lieb and Sabin \cite{Frank-Lieb-:Sabin:2016:extremisers-all-dimensions} prove that extremizers always exist for the Tomas-Stein inequality \eqref{TS} in all dimensions provided that a well-known conjecture the Strichartz inequalities for the Schr\"odinger equations is true.  They use the method of the missing mass.  Very recently, We have established a similar result by using profile decompositions and Tao's bilinear restriction estimates for paraboloids in \cite{Tao:2003:paraboloid-restri}.  

The work \cite{Christ-Shao:extremal-for-sphere-restriction-I-existence, Christ-Shao:extremal-for-sphere-restriction-II-characterizations, Shao:2015:extremals-1d-sphere} are partly motivated by the recent progress of application of the concentration compactness method or the profile decompositions in critical dispersive partial differential equations, see for instance Bourgain \cite{Bourgain:1999:radial-NLS}, Colliander, Keel, Staffilani, Takaoka and Tao \cite{I-teem:2008:GWP-for-energy-critical-NLS-in-R3}, Kenig and Merle \cite{Kenig-Merle:2006:focusing-energy-NLS-radial, Kenig-Merle:2008:focusing-energy-nonlinear-Wave} for radial or general data. In Lemma \ref{le-4} below establishing smoothness of extremizers, the analysis resembles some feature in the works of critical equations. We need to show that the critical points break the $L^2$ approximate scaling and hence gain certain regularity. 

In this note, we chacterize the extremizers in hope of finding the exact forms. We will prove that solutions to the following generalized Euler-Lagrange equation, which the extremizers satisfy, are smooth. The equation to the inequality \eqref{TS} is that, for $f\in L^2(S^1)$, 
\begin{equation}\label{Euler-Lagrange}
f\sigma \ast f\sigma \ast f\sigma \ast \tilde{f} \sigma \ast \tilde{f}\sigma (x) = \lambda f(x), \text{ for almost everywhere } x\in S^1,
\end{equation}
where $\lambda=\mathcal{R}^6 \|f\|^4_{L^2}$ and $\tilde{f} (x) =\bar{f}(-x)$, $\bar{f}$ denotes the complex conjugate of $f$. The main result is the following.
\begin{theorem}\label{thm-1}
Any $L^2$ solution to the Euler-Lagrange equation \eqref{Euler-Lagrange} is smooth on $S^1$.
\end{theorem}

Our proof of this theorem follows roughly the similar lines as in \cite{Christ-Shao:extremal-for-sphere-restriction-II-characterizations} by using the contraction mapping theorem. The first step is to show that solutions to the generalized Euler-Lagrange equation gain some regularity depending on the critical points themselves; the second step is a bootstrap argument upgrading the regularity to infinity, see Section \ref{sec:proof}. The difficulty is that there is no useful formula for the convolution $\sigma *\sigma *\sigma*\sigma*\sigma $, see also \cite[Section 2]{Carneiro-Foschi-Silva-Thiele:2015:sharp-trilinear-inequality-S1}. However it is uniformly bounded by a simple application of the Hausdorff-Young inequality, which is used in Lemma \ref{le-3}. Theorem \ref{thm-1} is also established by Diogo Oliveira e. Silva and Rene Quilodr\'an in \cite{Silva-Quilodran:2021:S1smoothness} by studying the H\"older estimate of the four-fold convolution $\sigma *\sigma *\sigma*\sigma$. Their estimates are used by us as auxiliary lemmas in Section \ref{sec:crappie}. 

This paper is organized as follows. In Section \ref{sec:notation}, we set up some notations. In Section \ref{sec:proof}, we give the main argument showing that the extremizers to \eqref{TS} are smooth.

{\bf Acknowledgement.} The author was supported in part by the NSF grant DMS-1160981 and KU 2016 -2017 general research fund.

\section{Notation}\label{sec:notation}
For $s\ge 0$, $H^s= H^s(S^1) $ denotes the usual Sobolev space of functions having $s\ge 0$ derivatives in $L^2$. We also write $H^0$ by $L^2$. Consider the action of the group $O(2)$ of all rotations of $\mathbb{R}^2$ acting on $S^1$. This action gives rise in a natural way to actions on functions by 
$$\Theta (f) = f \circ \Theta$$
and on finite Borel measures on $\mathbb{R}^2$ by $\Theta_\ast (\mu) (E) = \mu (\Theta(E))$.  The extension satisfies the basic identity 
$$\Theta_\ast (\mu \ast \nu) = \Theta_\ast (\mu) \ast \Theta_\ast (\nu). $$

Let $\{X_j: j =1,2\}$ be two $C^\infty$ vector fields on $S^1$ which generate rotations about the two coordinate axes, where $X_j$ is along the $x_j$ direction on $\mathbb{R}^2$; thus $\operatorname{exp} (tX_j)$, the exponential map acting on $X_j$ \cite[Page 130]{Petersen:2016:RiemannianGemeotry}, is obtained by rotating $x\in \mathbb{R}^2$ by $t$ radians about the $j$-th coordinate axis. These two vector fields span the one dimension tangent space to $S^1$ at each of its points. So $H^1(S^1)$ is equal to the set of all $f\in L^2(S^1)$ for all $X_j(f) \in L^2(S^1)$ for all indices $j\in \{1,2\}$. 

For $\alpha\in (0,1)$, we denote by $\Lambda^\alpha$ the space of all H\"{o}lder continuous functions of order $\alpha$ on $S^1$, with norm 
$$\|f\|_{\Lambda^\alpha} = \|f\|_{C^0}+ \sup_{x\neq y} \frac {|f(x)-f(y)|}{|x-y|^\alpha}.$$
For $\alpha\in (0,1)$, $\Lambda^\alpha$ equals the set of all continuous functions $f$ for which there exists $C<\infty$ such that $$|\operatorname{exp}(tX_j)f(x)- f(x)| \le C |t|^\alpha$$ 
for all $t\in \mathbb{R}$ and $x\in S^1$ for $j=1,2$, with a corresponding equivalence of norms. We denote by $\operatorname{Lip}(S^1)$ the space of all Lipschitz continuous functions from $S^1$ to $\mathbb{C}$, equipped with the norm $$\|f\|_{C^0} + \sup_{x\neq y} \frac {|f(x)-f(y)|}{|x-y|}. $$

For $0\le s \notin \mathbb{Z}$, we write $s= k+\alpha$, where $k\in \mathbb{Z}$ and $\alpha \in (0,1)$. For $s\in (0,1)$, we define $\mathcal{H}^s$ to be the set of all $f\in L^2(S^1)$ for which 
$$\|f\|_{\mathcal{H}^s} = \|f\|_{L^2(S^1)} + \sum_{j=1}^2 \sup_{0<|t|\le 1} \frac {\|\operatorname{exp}(tX_j) f- f \|_{L^2(S^1)}}{|t|^s}$$ 
is finite. For $s=0$, we define $\mathcal{H}^0=L^2$. For $s = k +\alpha$ with $k\in \mathbb{Z}^+$ and $\alpha\in (0,1)$, $\mathcal{H}^s$ is the set of all $f\in L^2(S^1)$ for which 
$$ \|f\|_{\mathcal{H}^s} = \|f\|_{L^2(S^1)} +\sum_Y \sum_{j=1}^2 \sup_{0<|t|\le 1} \frac {\|Yf\circ\operatorname{exp}(tX_j)f-Yf \|_{L^2(S^1)}}{|t|^s} $$
is finite, where $Y$ ranges over the finite set of all compositions $X_{i_1} \circ X_{i_2} \circ \cdots \circ X_{i_m}$ with $0\le m\le k$ factors. Here $f = Yf$, where $Y$ has zero factors. The mapping
$f\mapsto \Theta(f) = f \circ \Theta$ maps $\mathcal{H}^s$ isometrically to $\mathcal{H}^s$, uniformly for all $\Theta\in O(2)$. For any $0<t<s$, it is not hard to see that $\mathcal{H}^s$ is contained in the Sobolev space $H^t$, and 
\begin{equation}\label{eq-51}
\|f\|_{H^t} \le C(s,t) \|f\|_{\mathcal{H}^s}
\end{equation} for all $f\in \mathcal{H}^s$, see for instance \cite[Chapter 5, Proposition 10 and Theorem 5]{Stein:1970:singular integrals}. 

\section{the proof}\label{sec:proof}
In this section, we prove Theorem \ref{thm-1}. We first show that solutions to the generalized Euler-Lagrange equation \eqref{Euler-Lagrange} gain some regularity, see Lemma \ref{le-4}. Then we upgrade the regularity to infinity, see Lemma \ref{le-5}. We begin with a trivial interpolation result.
\begin{lemma}\label{le-interpolation}
For $0<\beta<\alpha$,
$$ \|f\|_{\mathcal{H}^\beta} \le C\|f\|^{1-\frac \beta\alpha}_{\mathcal{H}^0} \|f\|^{\frac \beta\alpha}_{\mathcal{H}^\alpha}\sim\|f\|^{1-\frac \beta\alpha}_{L^2} \|f\|^{\frac \beta\alpha}_{\mathcal{H}^\alpha}.$$
\end{lemma}
\begin{proof}
The inequality follows from
\begin{equation}\label{eq-5}
\|f\|_{L^2} \le \|f\|^{1-\frac \beta\alpha}_{L^2} \|f\|^{\frac \beta\alpha}_{\mathcal{H}^\alpha}
\end{equation}
and for $0<|t|\le 1$, 
\begin{align*}
\frac {\|\operatorname{exp}(tX_j)f-f\|_{L^2}}{|t|^\beta} & =\|\operatorname{exp}(tX_j)f-f\|_{L^2}^{1-\frac \beta\alpha} \left( \frac {\|\operatorname{exp}(tX_j)f-f\|_{L^2}}{|t|^\alpha}\right)^{\frac \beta\alpha} \\
& \le C \|f\|_{L^2}^{1-\frac \beta\alpha}\|f\|^{\frac \beta\alpha}_{\mathcal{H}^\alpha}.
\end{align*}
\end{proof}

\begin{lemma}\label{le-2}
Let $\mu=\sigma \ast \sigma \ast \sigma \ast \sigma \ast \sigma$. Then $\|\mu\|_{L^\infty (\{|x|\le 5\})} \le C$ for some constant $C>0$. 
\end{lemma}
\begin{proof}
Recall that for $0<|x|<2$, $f_0(x):= \sigma * \sigma(x) = \frac {4}{|x|\sqrt{4-|x|^2}}$, see for instance \cite{Foschi-Klainerman:2000:bilinear-space-time-estimates-for-wave}. This can be also computed by using the definition of the convolution of two circle measures, see for instance \cite[Page 489, Equation (67)]{Stein:1993}. Indeed, for any $0<\epsilon<\frac {1}{10}$,  let $A_\epsilon :=\{ y\in \mathbb{R}^2:\, 1-\epsilon \le |y|\le 1+\epsilon \}$, 
\begin{align*}
\sigma *\sigma (x) &:= \lim_{\epsilon\to 0} \frac {2|A_\epsilon \cap \{x-A_\epsilon \}|} {4\epsilon^2} \\
& = \lim_{\epsilon\to 0} \frac {8\epsilon^2}{4\epsilon^2 \sin 2\theta} \\
&= \frac {1}{\sin \theta\cos\theta  } = \frac {4}{|x|\sqrt{4-|x|^2}},
\end{align*}
where $\sin \theta = |x|/2$ and $\theta $ is the angle the vector $x/2$ is facing to in the right triangle. 
\begin{figure}[h]
        \centering
        \includegraphics[width=0.6\textwidth]{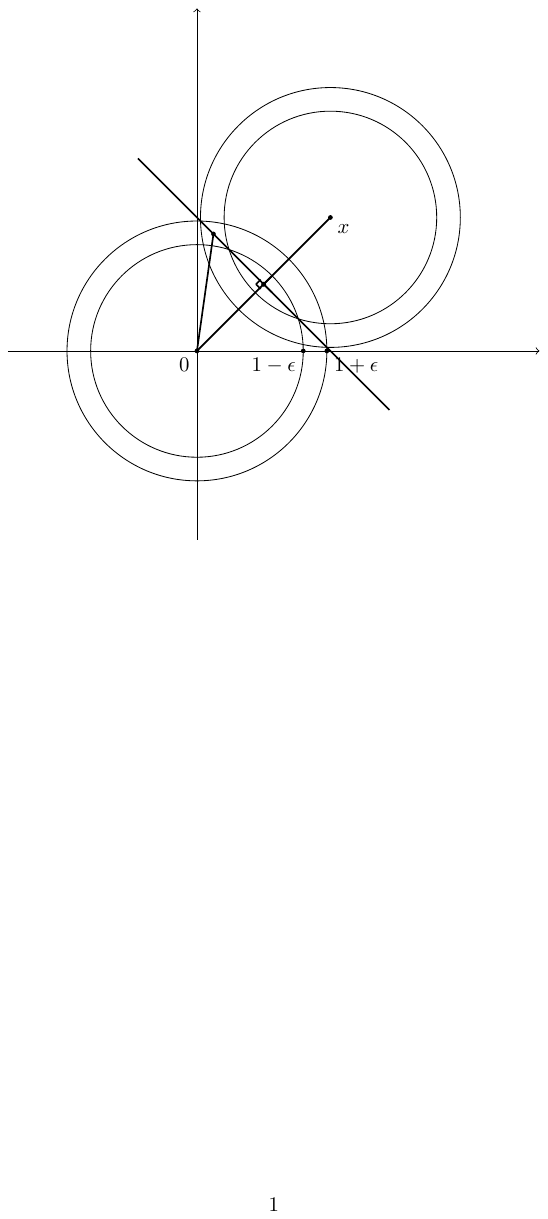}
        \caption{The convolution of two measures on the circle.}
        \label{fig-1}
\end{figure}

Then we have $$\mu= f_0*f_0*\sigma. $$
Then since $\|f_0\|_{L^1(\mathbb{R}^2)}<\infty,$ $\mu \in L^1(\mathbb{R}^2)$ by Funibi's theorem and Young's inequality. On the other hand, $\widehat{\mu} = \widehat{\sigma}^5 \in L^1$ from the decay estimate of $\widehat{\sigma}$, i.e., $|\widehat{\sigma }(x)| \le |x|^{-1/2}$,  for sufficiently large $|x|$. Thus from the Fourier inversion formula \cite[Corollary 1.21]{Stein-Weiss:1971:Fouerier Analysis}, we have 
$$ \mu(x) =\int e^{ix\cdot \xi } \widehat{\mu}(\xi) d\xi. $$
Thus an application of the $L^1\to L^\infty $ Hausdorff-Young inequality or the Riemann-Lebesgue lemma concludes the proof. 
\end{proof}

\begin{lemma}\label{le-3}
Suppose that $f_i\in \mathcal{H}^s$ for $i=1,\cdots, 5$ and $s\ge 0$. Then
$$\left\|f_1\sigma \ast f_2\sigma \ast f_3\sigma \ast f_4\sigma \ast f_5\sigma\right\|_{\mathcal{H}^s}\le C \prod_{i=1}^5 \|f_i\|_{\mathcal{H}^s}.$$
\end{lemma}
\begin{proof}
By the Cauchy-Schwarz inequality,
\begin{equation}\label{eq-10}
\left| f_1\sigma \ast\cdots \ast f_5\sigma(x)\right| \le \left( |f_1|^2\sigma \ast\cdots \ast |f_5|^2\sigma(x)\right)^{1/2}\left(\sigma \ast \cdots \ast \sigma (x)\right)^{1/2}.
\end{equation}
If we integrate both sides, by Lemma \ref{le-2},
\begin{equation}\label{eq-11}
\begin{split}
\|f_1\sigma \ast\cdots \ast f_5\sigma \|^2_{L^2(S^1)} &\le \int_{S^1} \left( |f_1|^2\sigma \ast\cdots \ast |f_5|^2\sigma(x) \right) \left|\sigma \ast \cdots \ast \sigma (x)\right| d\sigma\\
& \le\sup_{x\in S^1} \left|\sigma \ast \cdots \ast \sigma (x)\right| \int_{S^1} \left||f_1|^2\sigma \ast\cdots \ast |f_5|^2\sigma(x) \right| d\sigma\\
&\le C \prod_{i=1}^5 \|f_i\|^2_{L^2}.
\end{split}
\end{equation}
This proves the lemma when $s=0$.

Let $s>0$. For  $0<|t|\le 1$, we just need to prove
\begin{equation}\label{eq-12}
\begin{split}
&\frac {\|\operatorname{exp}(tX_j)\bigl( f_1\sigma \ast\cdots \ast f_5\sigma\bigr)-f_1\sigma \ast\cdots \ast f_5\sigma\|^2_{L^2}}{|t|^{2s}}\\
&=\int_{S^1} \frac {| \bigl(f_1\sigma \ast\cdots \ast f_5\sigma\bigr)\circ \operatorname{exp}(tX_j) (y) -\bigl(f_1\sigma \ast\cdots \ast f_5\sigma\bigr) (y)|^2}{|t|^{2s}} d\sigma(y)\\
&\le C \prod_{i=1}^5 \|f_i\|^2_{\mathcal{H}^s} .
\end{split}
\end{equation}
We compute that, for $j=1,2$ and $0<|t|\le 1$,
\begin{equation}\label{eq-14}
\begin{split}
& \bigl(f_1\sigma \ast\cdots \ast f_5\sigma\bigr)\circ \operatorname{exp}(tX_j) -\bigl(f_1\sigma \ast\cdots \ast f_5\sigma\bigr) \\
&=\bigl(f_1 \circ \operatorname{exp}(tX_j) -f_1 \bigr)\sigma \ast (\operatorname{exp}(tX_j) f_2)\sigma \ast (\operatorname{exp}(tX_j) f_3) \sigma\ast (\operatorname{exp}(tX_j) f_4)\sigma \ast (\operatorname{exp}(tX_j) f_5)\sigma+ \\
& +  (\operatorname{exp}(tX_j) f_1)\sigma \ast \bigl(f_2 \circ \operatorname{exp}(tX_j) -f_2 \bigr)\sigma \ast (\operatorname{exp}(tX_j) f_3) \sigma\ast (\operatorname{exp}(tX_j) f_4)\sigma \ast (\operatorname{exp}(tX_j) f_5)\sigma+\\
& +\cdots+ (\operatorname{exp}(tX_j) f_1)\sigma \ast (\operatorname{exp}(tX_j) f_2)\sigma  \ast (\operatorname{exp}(tX_j) f_3) \sigma\ast (\operatorname{exp}(tX_j) f_4)\sigma \ast \bigl(f_5 \circ \operatorname{exp}(tX_j) -f_5 \bigr).
\end{split}
\end{equation}
For the first term in \eqref{eq-14}, by the Cauchy-Schwarz inequality,
\begin{equation}\label{eq-15}
\begin{split}
&\bigl(f_1 \circ \operatorname{exp}(tX_j) -f_1 \bigr)\sigma \ast (\operatorname{exp}(tX_j) f_2)\sigma \ast (\operatorname{exp}(tX_j) f_3) \sigma\ast (\operatorname{exp}(tX_j) f_4)\sigma \ast (\operatorname{exp}(tX_j) f_5)\sigma \\
& \lesssim \left( |f_1 \circ \operatorname{exp}(tX_j) -f_1 |^2\sigma \ast |\operatorname{exp}(tX_j) f_2|^2\sigma \ast |\operatorname{exp}(tX_j) f_3|^2\sigma\ast |\operatorname{exp}(tX_j) f_4|^2\sigma* \right. \\
&\left.  \ast |\operatorname{exp}(tX_j) f_5|^2\sigma\right)^{1/2} \times \left( \sigma \ast \cdots\ast \sigma \right)^{1/2}.
\end{split}
\end{equation}
Applying the same reasoning to other terms in \eqref{eq-15} and going back to \eqref{eq-12}, we see the claim in Lemma \ref{le-3} for $s>0$ is proved. Thus we finish the proof of Lemma \ref{le-3}.
\end{proof}

Next we show that solutions to the Euler-Lagrange equation  \eqref{Euler-Lagrange} gain some regularity.
\begin{lemma}\label{le-4}
Suppose that $f\in L^2$ satisfies the Euler-Lagrange equation \eqref{Euler-Lagrange}. Then $f\in \mathcal{H}^s$ for some $s>0$. In particular, $f\in H^t$ for all $0\le t <s$.
\end{lemma}
\begin{proof}
For any $\epsilon>0$, we decompose $f$ such that $f=\phi_\epsilon+g_\epsilon$ such that $\|g_\epsilon\|_{L^2} <\epsilon$ and $\phi_\epsilon\in C^\infty$.

Recall that $$ \|\phi_\epsilon\|_{\mathcal{H}^s}=\|\phi_\epsilon\|_{L^2}+\sum_{j=1}^2\sup_{0<|t|\le 1} \frac {\|\operatorname{exp}(tX_j)f - f\|_{L^2(S^1)}}{ |t|^s},$$
and $$\|\phi_\epsilon\|_{\Lambda^s} =\|\phi_\epsilon\|_{L^\infty}+\sup_{x\neq y} \frac {|\phi_\epsilon(x)-\phi_\epsilon(y)|}{|x-y|^s}. $$
Then since $\phi_\epsilon \in C^\infty$,
\begin{equation}\label{eq-16}
\|\phi_\epsilon\|_{\mathcal{H}^s} \le C \|\phi_\epsilon\|_{\Lambda^s} <C_\epsilon<\infty.
\end{equation}
We remark that this bound depends on $\epsilon$.

From the Euler-Lagrange equation \eqref{Euler-Lagrange} and $f=\phi_\epsilon+ g_\epsilon$, we have
\begin{equation}\label{eq-17}
g_\epsilon = \mathcal{L}(\phi_\epsilon, g_\epsilon) +\mathcal{N}(\phi_\epsilon, g_\epsilon),
\end{equation}
where $\mathcal{L}$ is linear in $g_\epsilon$ and $\mathcal{N}$ is nonlinear in $g_\epsilon$. More precisely,
\begin{align*}
 \mathcal{L}&= -\phi_\epsilon+ \phi_\epsilon \sigma*\phi_\epsilon \sigma*\phi_\epsilon \sigma* \tilde\phi_\epsilon\sigma*\tilde \phi_\epsilon\sigma+ \\
&\qquad  +2 \phi_\epsilon \sigma*\phi_\epsilon \sigma* \phi_\epsilon \sigma*\tilde\phi_\epsilon \sigma* \tilde g_\epsilon\sigma+3 \phi_\epsilon \sigma*\phi_\epsilon \sigma*\tilde \phi_\epsilon \sigma*\tilde\phi_\epsilon \sigma*  g_\epsilon\sigma;
 \end{align*}
 and 
  \begin{align*} 
 \mathcal{N} &= g_\epsilon \sigma*g_\epsilon \sigma* g_\epsilon \sigma*\tilde g_\epsilon \sigma*\tilde g_\epsilon\sigma \\
 &+ 3 g_\epsilon \sigma*g_\epsilon \sigma* \tilde g_\epsilon \sigma*\tilde g_\epsilon \sigma* \phi_\epsilon\sigma+2 g_\epsilon \sigma*g_\epsilon \sigma* g_\epsilon \sigma*\tilde g_\epsilon \sigma*\tilde \phi_\epsilon\sigma +\\
 & +3g_\epsilon \sigma* \tilde g_\epsilon \sigma* \tilde g_\epsilon \sigma*\phi_\epsilon \sigma* \phi_\epsilon\sigma+6 g_\epsilon \sigma* \ g_\epsilon \sigma* \tilde g_\epsilon \sigma*\phi_\epsilon \sigma* \tilde \phi_\epsilon\sigma+\\
 &+ g_\epsilon \sigma* g_\epsilon \sigma*  g_\epsilon \sigma*\tilde \phi_\epsilon \sigma* \tilde \phi_\epsilon\sigma+\tilde g_\epsilon \sigma*\tilde g_\epsilon \sigma* \phi_\epsilon \sigma*\phi_\epsilon \sigma* \phi_\epsilon\sigma \\
 &+ 6 g_\epsilon \sigma*\tilde g_\epsilon \sigma* \phi_\epsilon \sigma*\phi_\epsilon \sigma*\tilde \phi_\epsilon\sigma+ 3 g_\epsilon \sigma*g_\epsilon \sigma* \phi_\epsilon \sigma*\tilde \phi_\epsilon \sigma*\tilde \phi_\epsilon\sigma.
 \end{align*}
For any $\alpha>0$,
\begin{equation}\label{eq-19}
\|\mathcal{L}(\phi_\epsilon, g_\epsilon)\|_{\Lambda^\alpha} \le \|\phi_\epsilon\|_{\Lambda^\alpha}+C \|\phi_\epsilon\|^5_{\Lambda^\alpha}+ C \|\phi_\epsilon\|^4_{\Lambda^\alpha} \|g_\epsilon\|_{L^2}
\end{equation}
Since $\|\phi_\epsilon\|_{\Lambda^\alpha}<C_\epsilon<\infty$ and $\|g_\epsilon\|_{L^2}\le \|f\|_{L^2}$,
\begin{equation}\label{eq-20}
\|\mathcal{L}(\phi_\epsilon, g_\epsilon)\|_{\Lambda^\alpha} <C_\epsilon<\infty.
\end{equation}
Together with $\|\mathcal{L}(\phi_\epsilon, g_\epsilon)\|_{\mathcal{H}^\alpha} \le \|\mathcal{L}(\phi_\epsilon, g_\epsilon)\|_{\Lambda^\alpha}$,  this implies
\begin{equation}\label{eq-21}
\|\mathcal{L}(\phi_\epsilon, g_\epsilon)\|_{\mathcal{H}^\alpha} \le C_\epsilon<\infty.
\end{equation}
On the other hand, by Lemma \ref{le-3},
\begin{equation}\label{eq-22}
\begin{split}
\|\mathcal{N}(\phi_\epsilon, g_\epsilon)\|_{\mathcal{H}^0}
&\lesssim \|g_\epsilon\|^5_{\mathcal{H}^0}+ \|g_\epsilon\|^4_{\mathcal{H}^0} \|\phi_\epsilon\|_{\mathcal{H}^0}+ \|g_\epsilon\|^3_{\mathcal{H}^0}\|\phi_\epsilon\|^2_{\mathcal{H}^0}+ \|g_\epsilon\|^2_{\mathcal{H}^0} \|\phi_\epsilon\|^3_{\mathcal{H}^0}\\
&\lesssim \epsilon^5 +\epsilon^4 +\epsilon^3+\epsilon^2\lesssim \epsilon^2,
\end{split}
\end{equation}
as $\|g_\epsilon\|_{\mathcal{H}^0}\sim \|g_\epsilon\|_{L^2} \le \epsilon$ and $\|\phi_\epsilon\|_{\mathcal{H}^0}\sim \|\phi_\epsilon\|_{L^2} \le 1.$ By the triangle inequality we have 
\begin{equation}\label{eq-22a}
\|\mathcal{L}(\phi_\epsilon, g_\epsilon)\|_{\mathcal{H}^0} \le C\epsilon.
\end{equation}

Choosing $\epsilon$ sufficiently small, and interpolating between \eqref{eq-21} and \eqref{eq-22a}, we see that there exists $s(\epsilon)$ depending on $\epsilon$ such that
\begin{equation}\label{eq-23}
\|\mathcal{L}(\phi_\epsilon, g_\epsilon)\|_{\mathcal{H}^{s(\epsilon)}}\lesssim \epsilon^{\frac 78}.
\end{equation}
From the two bounds $\|\phi_\epsilon\|_{\mathcal{H}^{0}}\le 1$ and $\|\phi_\epsilon\|_{\mathcal{H}^{\alpha}}<C(\epsilon)<\infty$, again choosing $s(\epsilon)$ sufficiently small, we see that
\begin{equation}\label{eq-24}
\|\phi_\epsilon\|_{\mathcal{H}^{s(\epsilon)}} <\epsilon^{-1/5}.
\end{equation}
Next we use the argument of Picard's iteration to show that $f$ will gain some regularity. Fixing the small $\epsilon>0$ above, we know that $g_\epsilon \in L^2$ and $\phi_\epsilon\in C^\infty.$ Define the iteration mapping and the ball in $\mathcal{H}^{s(\epsilon)}$,
\begin{equation}\label{eq-25}
\begin{split}
\mathcal{L}_\epsilon (h)& =\mathcal{L}(\phi_\epsilon, g_\epsilon)+\mathcal{N}(\phi_\epsilon, h), \\
\mathcal{B} & =B(\mathcal{L}(\phi_\epsilon, g_\epsilon), \epsilon^{\frac 34}).
\end{split}
\end{equation}
In the following two steps, we show that $\mathcal{L}_\epsilon$ is a contraction map on $\mathcal{B}$. The first step is to show that $\mathcal{L}_\epsilon$ maps $\mathcal{B}$ to itself. The second step is to show that $\mathcal{L}_\epsilon$ Lipschitz with the Lipschitz constant strictly less than $1$.

{\bf Step 1.} For any $h\in \mathcal{B}$, by the triangle inequality and \eqref{eq-23},
\begin{equation}\label{eq-26}
\|h\|_{\mathcal{H}^{s(\epsilon)}} \le \|h-\mathcal{L}(\phi_\epsilon, g_\epsilon)\|_{\mathcal{H}^{s(\epsilon)}} + \|\mathcal{L}(\phi_\epsilon, g_\epsilon)\|_{\mathcal{H}^{s(\epsilon)}} \lesssim \epsilon^{\frac 34}+ \epsilon^{\frac 78}
\lesssim \epsilon^{\frac 34}.
\end{equation}
Then similarly as in proving \eqref{eq-22}, by \eqref{eq-24},
\begin{equation}\label{eq-27}
\|\mathcal{N}(\phi_\epsilon, h)\|_{\mathcal{H}^{s(\epsilon)}} \lesssim \epsilon^{\frac 34\times 5}+\epsilon^{\frac 34 \times 4} \epsilon^{-\frac 15}+ \epsilon^{\frac 34 \times 3} \epsilon^{-\frac 15\times 2}+ \epsilon^{\frac 34 \times 2} \epsilon^{-\frac 15 \times 3}\le  \epsilon^{\frac 34}/10.
\end{equation}
Then for $h\in \mathcal{B}$,
\begin{equation}\label{eq-28}
\|\mathcal{L}_\epsilon (h)-\mathcal{L}(\phi_\epsilon, g_\epsilon )\|_{\mathcal{H}^{s(\epsilon)}} =
\|\mathcal{N}(\phi_\epsilon, h)\|_{\mathcal{H}^{s(\epsilon)}}  \le \epsilon^{\frac 34}.
\end{equation}
This proves that $\mathcal{L}_\epsilon$ is a map from $\mathcal{B}$ to $\mathcal{B}$.

{\bf Step 2.} We take $h_1, h_2 \in \mathcal{B}$. Then by \eqref{eq-26},
\begin{equation}\label{eq-29}
\|h_1\|_{\mathcal{H}^{s(\epsilon)}} \lesssim \epsilon^{\frac 34}, \text{ and } \|h_2\|_{\mathcal{H}^{s(\epsilon)}} \lesssim \epsilon^{\frac 34}.
\end{equation}
Note that by \eqref{eq-24}, $\|\phi_\epsilon\|_{\mathcal{H}^{s(\epsilon)}} \le \epsilon^{-\frac 15}$, then by Lemma \ref{le-3},
\begin{equation}\label{eq-30}
\begin{split}
&\mathcal{L}_\epsilon(h_2) -\mathcal{L}_\epsilon(h_1) = \mathcal{N}(\phi_\epsilon, h_2) -\mathcal{N}(\phi_\epsilon, h_1) \\
&\lesssim \|h_2-h_1\|_{\mathcal{H}^{s(\epsilon)}}\left(5 \epsilon^{\frac 34\times 4}+ 5\times 5 \epsilon^{\frac 34\times 3-\frac 15}+10\times 3 \epsilon^{\frac 34\times 2-\frac 15\times 2}+ 10\times 2 \epsilon^{\frac 34-\frac 15\times 3}\right).
\end{split}
\end{equation}
To conclude, if taking $\epsilon$ sufficiently small,
\begin{equation}
\|\mathcal{L}_\epsilon(h_2) -\mathcal{L}_\epsilon(h_1)\|_{\mathcal{H}^{s(\epsilon)}} \le \alpha \|h_2-h_1\|_{\mathcal{H}^{s(\epsilon)}}
\end{equation}
for some $0<\alpha<1$. So $\mathcal{L}_\epsilon$ is a contraction mapping on $\mathcal{B}$. Therefore there exists a unique $h_\epsilon\in \mathcal{B}\subset \mathcal{H}^{s(\epsilon)}$ such that
\begin{equation}\label{eq-31}
h_\epsilon=\mathcal{L}_\epsilon (h_\epsilon) =\mathcal{L}(\phi_\epsilon, g_\epsilon)+ \mathcal{N}(\phi_\epsilon, h_\epsilon).
\end{equation}
Moreover $\|h_\epsilon\|_{H^{s(\epsilon)}} \lesssim \epsilon^{\frac 34}$. When $H^{s(\epsilon)}$ is replaced by $L^2$, the same argument implies that there exists a unique solution in $L^2$. 
Since $\mathcal{H}^{s(\epsilon)} \subset \mathcal{H}^0=L^2$, if $\epsilon $ is sufficiently small, then $h_\epsilon$ is also the unique $L^2$ solution with small $L^2 $ norm. We know that in $L^2$ there holds $$g_\epsilon=\mathcal{L}(\phi_\epsilon, g_\epsilon)+ \mathcal{N}(\phi_\epsilon, g_\epsilon),$$
and $g_\epsilon$ has small $L^2$ norm.  So $g_\epsilon$ agrees with $h_\epsilon$ in $L^2$. This upgrades $g_\epsilon \in \mathcal{H}^{s(\epsilon)}$. It in turn shows that $f\in \mathcal{H}^{s(\epsilon)}$. Note that $s(\epsilon)$ depends on $f$.
\end{proof}

The second main ingredient is a bootstrap lemma. 

\begin{lemma}\label{le-5}
For any $\epsilon>0$, there exists $\delta >0$ such that for any $s \in [\epsilon, \infty) \setminus \mathbb{Z}$ and any function $f\in \mathcal{H}^s(S^1)$, then 
\begin{equation}\label{eq-32}
f\sigma* f\sigma *f\sigma *f\sigma *f\sigma |_{S^1} \in \mathcal{H}^t (S^1)  
\end{equation} for all $t\in [0,s+\delta]\setminus \mathbb{Z}$.

\end{lemma}
This proof is similar to \cite[Lemma 3.2]{Christ-Shao:extremal-for-sphere-restriction-II-characterizations} and so will be omitted. It relies on the following proposition, which is in the same spirit as \cite[Lemma 2.6]{Christ-Shao:extremal-for-sphere-restriction-II-characterizations}.

\begin{proposition}\label{prop-1}
For any $\epsilon>0$, there exists $\delta>0$ such that 
$$f_1\sigma \ast f_2\sigma \ast f_3 \sigma \ast f_4 \sigma *h\sigma \in \mathcal{H}^\delta $$
whenever $f_i \in \mathcal{H}^\epsilon(S^1), 1\le i\le 4$, and $h \in H^0(S^1)$, and 
\begin{equation}\label{eq-33}
\| f_1\sigma \ast f_2\sigma \ast f_3 \sigma \ast f_4 \sigma *h\sigma \|_{\mathcal{H}^\delta} \le C_\epsilon  \prod_{j=1}^4 \|f_i\|_{\mathcal{H}^\epsilon} \|h\|_{H^0}. 
\end{equation}
\end{proposition}
\begin{proof} Without loss of generality, we suppose that $$ \|f_i\|_{\mathcal{H}^\epsilon} =1, \text{ for }1\le i\le 4, \text{ and } \|h\|_{H^0}=1. $$
We divide the proof in the following 3 steps. 

{\bf Step 1.} Suppose that $f_i\in \operatorname{Lip}(S^1)$ for $1\le i\le 4$. By Theorem \ref{thm-completeZero}, there exists $\delta>0$, 
\begin{equation}\label{eq-33a}
\| f_1\sigma \ast f_2\sigma \ast f_3 \sigma \ast f_4 \sigma *h\sigma \|_{\mathcal{H}^\delta} \le C_\epsilon  \prod_{j=1}^4 \|f_i\|_{\operatorname{Lip}(S^1)} \|h\|_{H^0}. 
\end{equation}

{\bf Step 2. } For any $f\in \mathcal{H}^\epsilon(S^1)$ and $\eta >0$, there exists a decomposition that $f = f^\sharp+ f^b$, where $f^\sharp\in \operatorname{Lip}(S^1)$ and
\begin{align*}
\|f^b\|_{H^0} & \le \eta \|f\|_{\mathcal{H}^\epsilon}, \\
\|f^\sharp \|_{\operatorname{Lip}(S^1)} & \le \eta^{-C(\epsilon)} \|f\|_{\mathcal{H}^\epsilon}, \\
\|f^\sharp\|_{H^0} &\le C \|f\|_{\mathcal{H}^\epsilon},
\end{align*} 
where $C, C(\epsilon)$ independent of $f$. The existence of such decomposition follows from the inclusion that $\mathcal{H}^\epsilon \subset H^\tau$ for some $\tau =\tau(\epsilon)>0$, together with standard properties of $H^\tau$.  We perform such decompositions to each $f_i$, $1\le i\le 4$. {\bf Step 1} implies that 
\begin{equation}\label{eq-36} 
f_1^\sharp\sigma \ast f_2^\sharp\sigma \ast f_3^\sharp\sigma \ast f_4^\sharp\sigma \ast h\sigma  \in \mathcal{H}^\delta(S^1), 
\end{equation}
with the operator norm $O(\eta^{-4C(\epsilon)})$. On the other hand,  
\begin{equation}\label{eq-37}
\|f_1^b\sigma \ast f_2^b\sigma \ast f_3^b\sigma \ast f_4^b \sigma \ast h\sigma  \|_{L^2(S^1)}\le C \prod_{j=1}^4 \|f_i^b\|_{L^2(S^1)} \|h\|_{L^2} \le C \eta^4. 
\end{equation}
Similarly the contributions of the pairs $(f_i^\sharp, f_j^b)$, $1\le i,j\le 4$, belong to $L^2(S^1)$ with norms $O(\eta)$, since  $f_i^\sharp \in H^0$ is of $O(1)$. 

So far we have shown that for any $\eta>0$, $F= f_1\sigma \ast f_2 \sigma \ast f_3 \sigma \ast f_4 \sigma \ast h\sigma $ can be decomposed as the sum of two functions
\begin{equation}\label{eq-38}
F= F_\eta+ F^\eta,
\end{equation}
where $F_\eta\in L^2$ and $\|F_\eta\|_{L^2} \le \eta$, and $F^\eta \in \mathcal{H}^\delta$, and $\|F^\eta\|_{\mathcal{H}^\delta} \le C \eta^{-C(\epsilon)}$. Then we claim that $F\in \mathcal{H}^\delta$ for some $\delta $ depending on $\epsilon$. 

{\bf Step 3. } Let $0<|t|\le 1$ and $\eta>0$ be a parameter to be determined. For $F=F_\eta+F^\eta$, then 
\begin{equation}\label{eq-39}
\|\operatorname{exp}(tX_j) F^\eta-F^\eta\|_{L^2(S^1)} \le C |t|^\delta \|F^\eta\|_{\mathcal{H}^\delta }\le C |t|^\delta \eta^{-C(\epsilon)};
\end{equation}
and 
\begin{equation}\label{eq-40}
\|\operatorname{exp}(tX_j) F_\eta-F_\eta\|_{L^2(S^1)} \le  2 \|F_\eta\|_{L^2(S^1)}\le 2 \eta. 
\end{equation}
Then by the triangle inequality
\begin{equation}\label{eq-41}
\|\operatorname{exp}(tX_j) F-F\|_{L^2(S^1)} \le C |t|^\delta\eta^{-C(\epsilon)} + 2\eta. 
\end{equation}
Define $\eta$ by $C |t|^\delta\eta^{-C(\epsilon)} = 2\eta $. Then $$\eta= \left( \frac {C|t|^\delta}{2}\right)^{\frac 1{1+C(\epsilon)}}. $$
Therefore 
\begin{equation}\label{eq-42}
\|\operatorname{exp}(tX_j) F-F\|_{L^2(S^1)} \le 4 \left(\frac C2 \right)^{\frac 1{1+C(\epsilon)}} |t|^{\frac \delta{1+C(\epsilon)}} =C_\epsilon |t|^{\frac \delta{1+C(\epsilon)}} .
\end{equation}
We re-define $\delta$ to be $\frac {\delta}{1+C(\epsilon)}$. This finishes the proof of Proposition \ref{prop-1}.
\end{proof}

Therefore from Lemma \ref{le-4} and \ref{le-5}, the proof of Theorem \ref{thm-1} is complete. 

\section{The auxiliary lemmas}\label{sec:crappie}
In this section, we establish the H\"older estimate for the four-fold convolution of the circle measure as in \cite{Silva-Quilodran:2021:S1smoothness}. Given $\gamma\in (0,1)$ and let $H=|\cdot|^\gamma f_1\sigma*f_2\sigma *f_3\sigma *f_4\sigma:\mathbb{R}^2\to \mathbb{C} $ with $f_i\in \operatorname{Lip}(S^1), 1\le i\le 4$ supported on the ball $\{|x|\le 4\}$. The function $H$ satisfying that, for some $\tau\in (0,1)$ and $C<\infty$, 
\begin{equation}\label{eq-Holder} |H(x)-H(x')| \le C|x-x'|^\tau, \forall x, x'\in \{|x|\le 4\}\setminus \{0\}. 
\end{equation}
Then $H\in L^\infty (\mathbb{R}^2)$ and is continuous in $\{|x|\le 4\}\setminus \{0\}$. Given $\gamma \in (0,1)$, let $\mathcal{K}_\gamma = |\cdot |^{-\gamma } H$ and define the corresponding linear operator $\mathcal{K}_\gamma : C^0(S^{d-1}) \to L^2(S^{d-1})$ via 
\begin{equation}\label{eq-Caldeon-Zygmund} 
\mathcal{K}_\gamma f(\omega) = \int_{S^1} f(\nu) K_\gamma (\omega-\nu) d\sigma (\nu). 
\end{equation}
\begin{mainthm}\cite[Lemma 5.1]{Silva-Quilodran:2021:S1smoothness}\label{thm-completeZero}
Let $\gamma \in (0,1)$. Let $\mathcal{K}_\gamma$ be the linear operator defined in \eqref{eq-Caldeon-Zygmund}. Then there exists $\delta=\delta (\gamma)>0$ such that a bounded operator from $L^2(S^1)$ to $H^\delta (S^1)$.  
\end{mainthm}
The proof of Theorem \ref{thm-completeZero} relies on the following three theorems:. The first step is Theorem \ref{Caldeon-Zygmund2}, where we establish a convolution estimate for the circle measure. The second step is Theorem \ref{Calderon-Zygmund3-porkbelly}, where we establish an estimate on the 4-fold convolution estimate of the circle measures. The third step is Theorem \ref{Calderon-Zygmund4}, where we establish a H\"older estimate for the weighted 4-fold convolution estimate of the circle measures. Start by recalling that the two-fold convolution $\sigma*\sigma $ defines a measure supported on the ball $B_2:=\{|x|\le 2\} \subset \mathbb{R}^2$, which is absolutely continuous with respect to the Lebesgue measure on $B_2$, and whose Radon-Nikodym derivative equals, by a simple geometric computation, 
\begin{equation}\label{eq-metric}
\sigma*\sigma (x) = \frac {4}{|x|\sqrt{4-|x|^2}}. 
\end{equation} 
Let $h_1, h_2\in \operatorname{Lip} (S^1)$. The Radon-Nikodym derivative function $u_{12}$ defined by the relation $h_1\sigma *h_2\sigma (x) =u_{12} (x) \sigma*\sigma(x) $ for $0<|x| \le 2$ and $u_{12}(x) =0$ for $|x|>2$ can be expressed as 
$$ u_{12}(x) = \frac{1}{|\Gamma_x|}\int_{\Gamma_x} h_1(\nu) h_2(x-\nu) d\sigma_x(\nu),$$
where $\Gamma_x=S^1\cap (x+S^1)$, which consists of two points. Let $x^{\perp}$ be the $90^0$-counterclockwise rotation of $x$, so that $x^\perp \cdot x =0$ and $|x^\perp|= |x|$. Given $x\in B_2\setminus \{0\}\subset \mathbb{R}^2$, there exists unique-up-to-permutation $x_1,x_2\in S^1$, such that $x=x_1+x_2$. The vectors $x_1,x_2$ are explicitly given by 
$$x_1= \frac {x}{2}+(1-\frac {|x|^2}{4})^{1/2}\frac {x^\perp}{|x|}, x_2 = \frac x2-(1-\frac {|x|^2}{4})^{1/2} \frac {x^\perp }{|x|}. $$
Given $h_1, h_2 \in \operatorname{Lip}(S^1)$, the convolution $h_1\sigma *h_2\sigma $ can be written in the following way: If $0<|x|\le 2$, then 
$$ h_1\sigma *h_2\sigma (x) = 2 \dfrac {h_1(x_1)h_2(x_2) +h_1(x_2)h_2(x_1))}{|x|\sqrt{4-|x|^2}}$$
and for $|x|>2$ one obviously has that $h_1\sigma *h_2\sigma (x) =0$. In this case, 
$$u_{12}(x) = \frac 12 \left( h_1(x_1)h_2(x_2) +h_2(x_1)h_2(x_1)\right), \text{ if } 0<|x|\le 2. $$
\begin{mainthm}\cite[Lemma 4.1]{Silva-Quilodran:2021:S1smoothness}\label{Caldeon-Zygmund2}
Let $x,x'\in B_2\setminus \{0\}\subset \mathbb{R}^2$. Then by 
$$|u_{12}(x)-u_{12}(x')| \le C \|h_1\|_{\operatorname{Lip} (S^1)}  \|h_2\|_{\operatorname{Lip} (S^1)} \left( |x-x'|^{1/2} +\left|\frac {x}{|x|}-\frac {x'}{|x'|} \right|\right)$$
for some universal constant $C<\infty.$
\end{mainthm}
Our goal is to establish a H\"older-type estimate for the fourfold convolution $h_1\sigma*h_2\sigma *h_3\sigma*h_4\sigma$, where $\{h_j\}_{j=1}^4$ are Lipschitz functions on the unit circle $S^1$. Let 
\begin{align*}
u_{12}(x) := \frac {1}{2} \left( h_1(x_1)h_2(x_2) + h_1(x_2)h_2(x_1)\right)1_{B_2} (x), \\
u_{34}(x) := \frac {1}{2} \left( h_3(x_1)h_4(x_2) + h_3(x_2)h_4(x_1)\right)1_{B_2} (x). 
\end{align*}
Both of which satisfy the conclusion of Theorem \ref{Caldeon-Zygmund2}. We write 
$$ F(x) :=\sigma * \sigma (x) = \dfrac {4\times 1_{B_2(x)}}{ |x| \sqrt{4-|x|^2}}.  $$
We consider the following upper bound, 
$$\dfrac {1}{ |x| \sqrt{4-|x|^2}} =\frac {\sqrt{4-|x|^2}}{4|x|}+ \frac {\sqrt{|x|}} {4\sqrt{4-|x|^2}} \le \frac {1}{|x|}+\frac {1}{\sqrt{2-|x|}}, \forall  |x|\le 2,$$
and the estimate 
$$ \sigma^{*4} (x) \lesssim (1+\log |x|) 1_{B_4(x)}, \forall x\in \mathbb{R}^2. $$
See for instance \cite[Eq. (3.21)]{Silva-Quilodran:2021:Lower-S-maximizers}. This inequality in particular implies 
$$ |\cdot |^\beta \sigma ^{*4} \in L^\infty (\mathbb{R}^2), \forall \beta>0. $$
Setting $H_\gamma (x) = |x|^\gamma (u_{12} F) *(u_{34}F)(x)$, we then have $H_\gamma \in L^\infty(\mathbb{R}^2),$ for any $\gamma\in (0,1)$ and $\{h_j\}_{j=1}^4 \subset L^\infty (S^1).$ The following preparatory result quantifies the smallness of the function $\left(1_E(\sigma*\sigma)\right)*(\sigma*\sigma) $ for certain sets $E\subset \mathbb{R}^2$ of small Lebesgue measure. This estimate is bilinear in nature, which resembles Bourgain's estimate in \cite[Lemma 111]{Bourgain:1998:refined-Strichartz-NLS}.

\begin{mainthm}\cite[Lemma 4.5]{Silva-Quilodran:2021:S1smoothness}\label{Calderon-Zygmund3-porkbelly}
Set $F=\sigma*\sigma $. Let $x\in B_4\subset \mathbb{R}^2$. Then for every $\gamma \in (0,1)$ and $s\in (0, \frac{\gamma}{2(\gamma+1)} ]$, there exists a constant $C_{\gamma, s} <\infty$ such that for all $\epsilon\in (0,1)$, 
$$ |x|^\gamma \int_{A(x,\epsilon)} F(y) F(x-y) dy \le C_{\gamma, s} \epsilon^{\min\{\frac 16, \frac {\gamma}{2(\gamma+1)}-s\}}, $$
$$ |x|^\gamma \int_{B_2\cap B(x,\epsilon)} F(y) F(x-y) dy \le C_{\gamma, s} \epsilon^{ \min\{\frac 12, \gamma-s\} }, $$
where $A(x,\epsilon) :=\{y\in B_2:\, 2-\epsilon \le |x-y|\le 2\}$ and $B(x,\epsilon) :=\{y\in B_2:\, |x-y|\le 2\}$.
\end{mainthm}

This in turns establishes the following theorem. 
\begin{mainthm}\cite[Proposition 4.6]{Silva-Quilodran:2021:S1smoothness}\label{Calderon-Zygmund4}
Given $\gamma\in (0,1)$ and $\{h_j\}_{j=1}^4 \subset \operatorname{Lip} (S^1)$, let $H_\gamma = |\cdot|^\gamma (h_1\sigma *h_2\sigma*h_3\sigma*h_4\sigma)$. Then there exists $\tau>0$ and $C<\infty$ such that for every $x,x'\in \mathbb{R}^2$, 
$$|H_\gamma (x) - H_\gamma (x')| \le C |x-x'|^\tau, $$
where $C \le C_0 \Pi_{j=1}^4 \|h_j\|_{\operatorname{Lip} (S^1)}$ for some constant $C_0<\infty$ depending only on $\gamma$. 
\end{mainthm}

\end{document}